\documentclass[preprint,12pt]{elsarticle}

\usepackage{amssymb}
\usepackage{amsmath}
\usepackage{amsthm}

\usepackage{orcidlink}

\journal{Journal of Mathematical Analysis and Applications }

\usepackage{hyperref}

\usepackage{xcolor}

\begin{document}
\sloppy

\newcommand{\Lip}{\operatorname{Lip}}
\newcommand{\lip}{\operatorname{lip}}
\newcommand{\LIP}{\operatorname{\mathbb{L}ip}}
\newcommand{\Int}{\operatorname{Int}}

\newtheorem{theorem}{Theorem}[section]
\newtheorem{lemma}[theorem]{Lemma}
\newtheorem{proposition}[theorem]{Proposition}
\newtheorem{corollary}[theorem]{Corollary}
\newtheorem{remark}[theorem]{Remark}

\theoremstyle{definition}
\newtheorem{definition}{Definition}
\newtheorem{example}{Example}
\newtheorem{problem}{Problem}

\begin{frontmatter}


 
\title{Takagi--van der Waerden functions in  metric spaces and their Lipschitz derivatives\tnoteref{Grant}}
\tnotetext[Grant]{
    The research was supported by the University of Silesia in Katowice, Mathematics Department (Iterative Functional Equations
and Real Analysis program)
	}

\author[us,chnu]{Oleksandr V. Maslyuchenko \orcidlink{0000-0002-1493-9399}\corref{CorrAuthor}} 
\ead{ovmasl@gmail.com}
\cortext[CorrAuthor]{Corresponding author}
\author[us]{Ziemowit M. Wójcicki \orcidlink{0009-0005-6433-0708}}
\ead{ziemo1@onet.eu}
\affiliation[us]
{organization={Institute of Mathematics, University of Silesia in Katowice},
addressline={Bankowa 12}, 
            city={Katowice},
            postcode={40-007}, 
            country={Poland}}
\affiliation[chnu]
{organization={Department of Mathematics and Informatics, Yuriy Fedkovych Chernivtsi National University},
	addressline={Kotsiubynskoho  2}, 
	city={Chernivtsi},
	postcode={58012}, 
	country={Ukraine}}
\begin{abstract}
We introduce the Takagi--van der Waerden function with parameters $a{>}b{>}0$  by setting $f_{a,b}(x)=\sum\limits_{n=1}^\infty b^n d\big(x,S_n\big)$, where $S_n$ is a maximal  $\frac1{a^n}$-separated set in a metric space $X$. So, if $X=\mathbb R$ and $S_n=\frac1{a^n}\mathbb Z$ then $f_{2,1}$ is the Takagi function and $f_{10,1}$ is the van der Waerden function  which are the famous examples of nowhere differentiable functions. Then we prove that the big Lipschitz derivative $\mathrm{Lip} f_{a,b}(x)=+\infty$ if $a>b>2$ and $x$ is a non-isolated point of $X$. Moreover, if the shell porosity $p^s(X,x)<\lambda<1$ for some $\lambda$ and each non-isolated point $x\in X$ then the little Lipschitz derivative $\mathrm{lip} f_{a,b}(x)=+\infty$ for large enough $a>b$ and any non-isolated point $x\in X$. In particular, this is true for any normed space. Finally, we prove that for any open set $A$ in a metric (normed) space $X$ without isolated points there exists a continuous function $f$ such that $\mathrm{Lip} f(x)=+\infty$ (and $\mathrm{lip} f(x)=+\infty$) exactly on $A$.
\end{abstract}




\end{frontmatter}



\section{Introduction}

The classical Rademacher theorem \cite[Theorem 3.2]{EvGa} asserts that every Lipschitz mapping $f:\mathbb{R}^m\to\mathbb{R}^n$ is differentiable almost everywhere.  Actually,  Rademacher's result \cite[Satz I]{Ra} and its generalization \cite{St} concern the two-dimensional case  but they are applicable not only for Lipschitz functions. In particular, the Stepanoff theorem \cite[Satz]{St} asserts that a continuous function $f$ is  almost everywhere differentiable on the complement of the set  ${L^\infty(f)=\big\{x:\Lip f(x)=\infty\big\}}$
where $\Lip f(x)=\limsup\limits_{u\to x} \frac{1}{\|u-x\|}\big\|f(u)-f(x)\big\|$ is the, so called, \textit{big Lipschitz derivative} (or, shortly, \textit{big lip}). Malý \cite{Mal} gave a simple proof of the Stepanoff theorem for a function $f:X\to\mathbb{R}^n$, where $X\subseteq\mathbb{R}^m$. Besides, there are many generalizations of the Rademacher theorem for Banach spaces (see, for example, \cite{LiPr,MalZa}). If we consider the \textit{little Lipschitz derivative} (or, shortly, \textit{little lip}) $\lip f(x)=\liminf\limits_{r\to 0}\sup\limits_{\|u-x\|<r}\frac1r\big\|f(u)-f(x)\big\|$ instead of $\Lip f$ then $f$ need not to be almost everywhere differentiable on the complement of $\ell^\infty(f)=\big\{x:\lip f(x)=\infty\big\}$ (see \cite[Theorem 1.4]{BaCs}). 
It is well-known that every differentiable function is \textit{pointwise Lipschitz} (that is $\Lip f<+\infty$). But $C^1$-functions have a stronger property: every such function is \textit{locally Lipschitz}. This means that its \textit{local Lipschitz derivative} $\LIP f(x)=\limsup\limits_{(u,v)\to (x,x)}\frac{1}{\|u-v\|}\big\|f(u)-f(v)\big\|<+\infty$ for any $x$.
So, the set $\mathbb L^\infty(f)=\big\{x:\LIP=\infty\big\}$ is also interesting for us.

In this paper we deals with the following general problem.

\begin{problem}[\textit{$\ell$-$L$-$\mathbb L$-problem}]\label{prob:ellLbL1}
	Let $X$ be a metric space. For which sets $A,B,C\subseteq X$ does there exist a continuous function $f\colon X\to\mathbb{R}$ such that $\ell^\infty(f)=A$, $L^\infty(f)=B$  and $\mathbb{L}^\infty(f)=C$?
\end{problem}
\noindent Also we may  consider some partial versions of this problem by omitting one or two of the sets $\ell^\infty(f)$, $L^\infty(f)$ and $\mathbb{L}^\infty(f)$. For example, the \textit{$\ell$-problem} means the description of the set $\ell^\infty(f)$ for a continuous function $f$; the \textit{$\ell$-$L$-problem} means the description of the sets $\ell^\infty(f)$ and $L^\infty(f)$ and so on. In the case where $X=\mathbb R$ the $L$-problem  was solved in \cite[Theorem 3.35 and Lemma 2.4]{BuHaRmZu} where the authors proved that $A$  has the property $A=L^\infty(f)$ (even $A=\ell^\infty(f)=L^\infty(f)$) for some continuous function $f$ if and only if $A$ is a $G_\delta$-set. Moreover, they proved  in \cite[Theorem 3.1]{BuHaRmZu} that for any $F_\sigma$-set $A$ there is a continuous function $f$ with $\ell^\infty(f)=A$. But the obvious necessary condition is that $\ell^\infty(f)$ is $F_{\sigma\delta}$. So, the $\ell$-problem is still open even for $\mathbb R$. Hanson \cite{Ha} solved $\ell$-$L$-problem in the case where $A=\ell^\infty(f)=\varnothing$. Another partial answer to the $\ell$-problem one can find in the recent paper \cite{RmZu} where the author proved that for any $F_{\sigma\delta}$-set $A\subseteq\mathbb R$ of Lebesgue measure zero there is an increasing absolutely continuous function $f\colon\mathbb R\to\mathbb R$ such that $\ell^\infty(f)=A$. We are also interested in the following more general problem.
\begin{problem}[\textit{$\lip$-$\Lip$-$\LIP$-problem}] Let $X$ be a metric space. For which functions $u,v,w\colon X\to\mathbb{R}$ does there exists a continuous function $f\colon X\to\mathbb{R}$ such that $\lip f =u$, $\Lip f=v$ and $\LIP f=w$?
\end{problem}

Of course, we may omit some of the functions $\lip f$, $\Lip f$, $\LIP f$ and consider simpler versions of the previous problem like\textit{ $\lip$-problem}, \textit{$\lip$-$\Lip$-problem}  and so on. In \cite{BuHaMaVe} the authors solved the $\lip$-problem for a function $u=\mathbf{1}_E$ where $E\subseteq\mathbb R$ and characterized the, so called, $\lip 1$ subsets of $\mathbb R$.

As we have seen, the current results in this direction concern the case ${X=\mathbb R}$. 
Therefore, we are going to extend some of the previous results to general metric spaces.
Our approach to this problem is based on a generalization of the 
Takagi-van der Waerden function to a metric space. 
The \textit{classical Takagi-van der Waerden function of type} 
$a>1$ (see, for example, \cite{Allaart}) is, by definition, the
function $f_a(x)=\sum\limits_{n=1}^\infty d\big(x,\frac{1}{a^n}\mathbb Z\big)$,  
where $d(x,E)=\inf\limits_{y\in E}|x-y|$ for any $x\in \mathbb R$ and
$E\subseteq\mathbb R$. In the case $a=2$ or $a=10$ we obtain two 
famous examples of nowhere differentiable functions which were 
constructed independently by Takagi and van der Waerden. 
We replace the set $\frac{1}{a^n}\mathbb Z$ by a maximal 
\mbox{$\frac{1}{a^n}$-separated} set $S_n$ in a metric space $X$, 
add a parameter $b<a$ and define a \textit{Takagi-van der Waerden function of type} $(a,b)$ as $f_{a,b}(x)=\sum\limits_{n=1}^\infty b^n d\big(x,S_n\big)$, $x\in X$. Then we prove that $\Lip f_{a,b}(x)=+\infty$ if $a>b>2$ and $x$ is a non-isolated point of $X$. To calculate the little Lipschitz derivatives we need some special property of o metric space which we call \textit{hermeticity}. It means that the \textit{shell porosity} (see \cite{Va}) $p^s(X,x)\le\lambda<1$ for some $\lambda$ and any  non-isolated
point $x$ in $X$. In particular, every normed space is hermetic. We prove that for any hermetic space $X$ there are $a>b>1$ such that $\lip f_{a,b}(x)=+\infty$ for any non-isolated point $x\in X$. So, in the last section we give a partial answer to the $\ell$-$L$-$\mathbb L$-problem for a hermetic space $X$, open sets $A=B$ and $C=\overline{A}$.

\section{Lipschitz derivatives}

Let $X$ be a metric space, $a\in X$ and $\varepsilon>0$. We always denote the metric on $X$ by $|\,\cdot\,-\,\cdot\,|_X$ and
\begin{align*}
	B(a,\varepsilon)&=B_X(a,\varepsilon)=\big\{x\in X\colon|x-a|_X<\varepsilon\big\},\\
	B[a,\varepsilon]&=B_X[a,\varepsilon]=\big\{x\in X\colon|x-a|_X\le\varepsilon\big\}.
\end{align*}

\begin{definition}
	Let $X$ and $Y$ be metric spaces, $f:X\to Y$ be a function, $x\in X$. Denote
	\begin{itemize}
		\item $\|f\|_{\lip}=\sup\limits_{u\ne v\in X}\frac1{|u-v|_X}\big|f(u)-f(v)\big|_Y$
		\item $\LIP f(x)=\limsup\limits_{(u,v)\to (x,x)}\frac1{|u-v|_X}\big|f(u)-f(v)\big|_Y$
		\item $\Lip f(x)=\limsup\limits_{u\to x}\frac1{|u-x|_X}\big|f(u)-f(x)\big|_Y$;
		\item $\lip f(x)=\liminf\limits_{r\to0^+}\sup\limits_{u\in B(x,r)}\frac1r\big|f(u)-f(x)\big|_Y$;
	\end{itemize}
The number $\|f\|_{\lip}$ is \emph{Lipschitz constant of $f$}.  The functions $\LIP f$,  $\Lip f$  and $\lip f$ are called the \emph{local, big and little Lipschitz derivative} respectively.
\end{definition}
We denote by $X^d$ the set of all non-isolated points of $X$. During the whole paper we assume that $\sup\varnothing=0$. As a consequence of this assumption we have $\LIP f(x)=\Lip f(x)=\lip f(x)=0$ for any $x\in X\setminus X^d$. 

Obviously, if $Y$ is a normed space then  $\|\cdot\|_{\lip}$ is an extended seminorm on $Y^X$ in the sense \cite{SaTaGa}. Moreover, $\|\cdot\|_{\lip}$ is a norm on the space all  Lipschitz functions $f\colon {X\to Y}$ vanishing at some fixed point in $X$.

We introduce some auxiliary notations:
	\begin{itemize}
		\item $\LIP^rf(x)=\big\|f|_{B(x,r)}\big\|_{\lip}=\sup\limits_{u\ne v\in B(x,r)}\frac1{|u-v|_X}\big|f(u)-f(v)\big|_Y$
		\item $\Lip^r f(x)=\sup\limits_{u\in B(x,r)}\frac1r\big|f(u)-f(x)\big|_Y, \Lip^r_+f(x)=\sup\limits_{u\in B[x,r]}\tfrac1r\big|f(u)-f(x)\big|_Y$
		\item $\Lip_r f(x)=\sup\limits_{0<\varrho<r}\Lip^\varrho f(x),\ \ \ \ \ \lip_r f(x)=\inf\limits_{0<\varrho<r}\Lip^\varrho f(x)$;
	\end{itemize}
Therefore, the definitions of the Lipschitz derivatives might be rewritten as follows
\begin{align}
	\label{eqn:ILip_definition}
        \LIP f(x)&=\inf\limits_{r>0}\LIP^rf(x)\\
        \label{eqn:lip_definition}
        \lip f(x)&=\liminf\limits_{r\to0^+}\Lip^r f(x)
\end{align}

Some authors (see, for example, \cite{Ha,BuHaMaVe,BuHaRmZu}) define  $\Lip f$  and $\lip f$ use the function
$\Lip^r_+f$ instead of $\Lip^rf$. In the case where $X$ is a normed space we have $B[x,r]=\overline{B(x,r)}$. 
Therefore, $\Lip^rf(x)=\Lip^r_+f(x)$ for any continuous function $f$.
But the previous equality does not hold for the discrete metric on $X$, nonconstant $f$ and $r=1$. However, we have the following.

\begin{proposition}
    Let $X$ and $Y$ be metric spaces and $f\colon X\to Y$ be a function.
    Then, for any non-isolated point $x\in X$, the following equalities hold \[
        \Lip f(x)=\limsup_{r\to 0^{+}}\Lip^rf(x)=\limsup_{r\to 0^{+}}\Lip_{+}^rf(x).
    \]
\end{proposition}
\begin{proof}
    Denote 
    \begin{align*}
        \alpha(r) &= \sup_{0<\rho<r}\Lip^{\rho} f(x), \\
        \beta(r)  &= \sup_{0<\rho<r}\Lip_{+}^{\rho} f(x), \\
        \gamma(r) &= \sup\left\{\frac{\vert f(u)-f(x)\vert_Y }{\vert u - x\vert_X}\colon u\in X, \; 0<\vert u-x\vert_X<r\right\}.
    \end{align*}
    Since $B(x,r)\subseteq B[x,r]$, we have $\alpha(r)\leq\beta(r)$. Next, we have
    \begin{align*}
        \beta(r) &= \sup_{0<\rho<r}\sup_{0<\vert u-x\vert_X\leq\rho}\tfrac{1}{\rho}\vert f(u)-f(x)\vert_Y \\
                 &\leq \sup_{0<\rho<r}\sup_{0<\vert u-x\vert_X\leq\rho}\tfrac{1}{\vert u-x\vert_X}\vert f(u)-f(x)\vert_Y = \gamma(r).
    \end{align*}
    On the other hand, we have
    \begin{align*}
        \gamma(r) &= \sup_{0<\rho<r}\sup_{\vert u-x\vert_X=\rho}\tfrac{1}{\rho}\vert f(u)-f(x)\vert_Y \\
                  &\leq \sup_{0<\rho<r}\sup_{0<\vert u-x\vert_X\leq\rho}\tfrac{1}{\rho}\vert f(u)-f(x)\vert_Y \\
                  &= \sup_{0<\rho<r}\Lip_{+}^{\rho}f(x) = \beta(r).
    \end{align*}
    We have shown, that $\alpha(r)\leq\beta(r)=\gamma(r)$ for $r>0$.
    It remains to show that $\alpha(r)=\beta(r)$, $r>0$.
    Let us assume that there exists $r_0>0$, such that $\alpha(r_0)<\beta(r_0)$.
    Denote $\varphi(r)=\Lip^r f(x)$ and $\psi(r)=\Lip_{+}^r f(x)$. Observe, that
    \begin{equation}\label{eq:some_nonsense}
        \rho\varphi(\rho)\leq\rho\psi(\rho)\leq r\varphi(r), \text{ for }  0<\rho<r.
    \end{equation}
    Since $\beta(r_0)=\sup\limits_{r<r_0}\psi(r)>\alpha(r_0)$, 
    there exists $r_1<r_0$ such that $\alpha(r_0)<\psi(r_1)$.
    Let $\varepsilon=\psi(r_1)-\alpha(r_0)>0$. Then, for any $r<r_0$,
    $\varphi(r)+\varepsilon\leq\alpha(r_0)+\varepsilon=\psi(r_1)$, so
    \begin{equation}\label{eq:some_more_nonsense}
        r_1\varphi(r)+r_1\varepsilon\leq r_1\psi(r_1)\leq r\varphi(r), \text{ for } r_1<r<r_0,
    \end{equation}
    where the second inequality follows from \eqref{eq:some_nonsense}.
    Note that
    \[
        \varphi(r)=\tfrac{1}{r}r\varphi(r)\leq\tfrac{1}{r_1}r_0\varphi(r_0), \text{ for } r_1<r<r_0,
    \]
    so, the function $\varphi$ is bounded on the interval $(r_1;r_0)$.
    However, by \eqref{eq:some_more_nonsense}, we have
    \[
        0 = \lim_{r\to r_{1}^{+}}(r-r_1)\varphi(r)\geq r_1\varepsilon >0,
    \]
    which is impossible. We have $\alpha(r)=\beta(r)=\gamma(r)$ for $r>0$.
    But $$\displaystyle\limsup_{r\to 0^{+}}\Lip^r f(x)=\inf_{r>0}\alpha(r),$$
    $$\displaystyle\limsup_{r\to 0^{+}}\Lip_{+}^r f(x)=\inf_{r>0}\beta(r)$$
    and $$\Lip f(x)=\inf\limits_{r>0}\gamma(r)$$ and the proof is finished.
\end{proof}

Note, that
\begin{align}
	\label{eqn:Lip_monotone}
	\Lip_r f(x)&\le \Lip_{r'}f(x)
	\ \ \text{and}\ \ 
	\lip_r f(x)\ge \lip_{r'}f(x) \text{ if }0<r<r',
\end{align}
So, the definitions and the previous proposition yield
\begin{align}
	\label{eqn:Lip_as_limsup}\Lip f(x)&=\inf\limits_{r>0}\Lip_rf(x)=\lim\limits_{r\to0^+}\Lip_rf(x),\\
	\label{eqn:lip_as_sup_lip_r}\lip f(x)&=\sup\limits_{r>0}\lip_rf(x)=\lim\limits_{r\to0^+}\lip_rf(x).
\end{align}
Therefore, it is easy to see that the following inequalities hold.
\begin{align}
	 \label{eqn:Lip_r_estimations}\lip_r f(x)&\le \Lip_r f(x)\le\LIP^r f(x)\text{ for any }r>0,\\
	\label{eqn:Lip_estimations}\lip f(x)&\le \Lip f(x)\le\LIP f(x).
\end{align}

\begin{definition}
	Let $X$ and $Y$ be metric spaces and $\gamma\ge 0$. A function ${f\colon X\to Y}$ is called
	\begin{itemize}
		\item  \emph{$\gamma$-Lipschitz}  if $\|f\|_{\lip}\le\gamma$;
		\item \emph{Lipschitz} if $\|f\|_{\lip}<\infty$;
		\item \emph{locally Lipschitz} if $\LIP f<\infty$. 
		\item \emph{pointwise Lipschitz} if $\Lip f<\infty$;
		\item \emph{weakly pointwise Lipschitz} if $\lip f<\infty$.
	\end{itemize} 
	Denote
	\begin{itemize}
		\item $\mathbb{L}(f)=\big\{x\in X:\LIP f(x)<\infty \big\}$;
		\item $\mathbb{L}^\infty(f)=\big\{x\in X:\LIP f(x)=\infty \big\}=X\setminus \mathbb{L}(f)$;
		\item $L(f)=\big\{x\in X:\Lip f(x)<\infty \big\}$;
		\item $L^\infty(f)=\big\{x\in X:\Lip f(x)=\infty \big\}=X\setminus L(f)$;
		\item $\ell(f)=\big\{x\in X:\lip f(x)<\infty \big\}$;
		\item $\ell^\infty(f)=\big\{x\in X:\lip f(x)=\infty \big\}=X\setminus\ell(f)$;
	\end{itemize}
\end{definition}

Inequalities (\ref{eqn:Lip_estimations}) yield the next assertion.

\begin{proposition}\label{prop:properties_Lip_lip}
	Let $X$ and $Y$ be metric spaces, and $f:X\to Y$ be a function. Then $\mathbb{L}(f)\subseteq L(f)\subseteq\ell(f)$ and $\ell^\infty(f)\subseteq L^\infty(f)\subseteq\mathbb{L}^\infty(f)$.
\end{proposition}

\section{Classification of the Lipschitz derivatives and pairs of Hahn}

Now we pass to the investigation of the type of semicontinuity of Lipschitz derivatives of continuous functions. 
In \cite{BuHaRmZu} semicontinuity of Lipschitz derivatives of a continuous function 
$f:\mathbb{R}\to\mathbb{R}$ was obtained from the continuity of $\Lip^rf$. 
But in the general situation this function need not be continuous. 
Therefore, we prove semicontinuity of  Lipschitz derivatives directly by the definitions.

\begin{definition}
	Let $X$ be a topological space and $f,g:X\to\overline{\mathbb{R}}$. We say that
	\begin{itemize}
		\item $f$ is \emph{lower semicontinuous} if  $f^{-1}\big((\gamma;+\infty]\big)$ is an open set for any $\gamma\in\overline{\mathbb{R}}$;
		\item $f$ is \emph{upper semicontinuous} if $f^{-1}\big([-\infty;\gamma)\big)$ is an open set for any $\gamma\in\overline{\mathbb{R}}$;
		\item $f$ is \emph{$F_\sigma$-lower semicontinuous} if  $f^{-1}\big((\gamma;+\infty]\big)$ is an $F_\sigma$-set for any ${\gamma\in\overline{\mathbb{R}}}$;
		\item $f$ is \emph{$F_\sigma$-upper semicontinuous} if $f^{-1}\big([-\infty;\gamma)\big)$ is an $F_\sigma$-set for any ${\gamma\in\overline{\mathbb{R}}}$;
		\item $(f,g)$ is \emph{a pair of Hahn}  if $f\le g$, $f$ is upper semicontinuous and $g$ is lower semicontinuous;
		\item $(f,g)$ is \emph{an $F_\sigma$-pair of Hahn} if $f\le g$, $f$ is  $F_\sigma$-lower semicontinuous and $g$ is $F_\sigma$-upper semicontinuous;
	\end{itemize}
\end{definition}
\begin{proposition}\label{prop:Lebesgue_clas_of_semicontinuous}
	Let $X$ be a topological space, $f:X\to\overline{\mathbb{R}}$ be an \mbox{($F_\sigma$-)upper} semicontinuous, $g:X\to\overline{\mathbb{R}}$ be an ($F_\sigma$-)lower semicontinuous and $\gamma\in\overline{\mathbb{R}}$. Then $f^{-1}\big([-\infty,\gamma]\big)$, $g^{-1}\big([\gamma, +\infty]\big)$ are 
	$G_\delta$-sets (resp. $F_{\sigma\delta}$-sets) and 
	$f^{-1}\big((\gamma,+\infty]\big)$, $g^{-1}\big([-\infty,\gamma)\big)$ are  $F_\sigma$-sets (resp. $G_{\delta\sigma}$-sets).
\end{proposition}
\begin{proof}
	Let $\gamma<+\infty$ and $\gamma_n\downarrow \gamma$. Since $f^{-1}\big([\gamma_n;+\infty]\big)$ are closed (resp.{ $G_\delta$-set}), we conclude that $f^{-1}\big((\gamma;+\infty]\big)=\bigcup\limits_{n=1}^\infty f^{-1}\big([\gamma_n;+\infty]\big)$ is $F_\sigma$-sets (resp. $G_{\delta\sigma}$-sets). The proof of the rest assertions is analogical.
\end{proof}

\begin{proposition}\label{prop:inf_and_sup_of_semicontinuous}
	Let $X$ be a topological space, $f_n:X\to\overline{\mathbb R}$ be an upper (lower) semicontinuous function for any $n\in\mathbb{N}$ and $f:X\to\overline{\mathbb{R}}$ be a function such that $f(x)=\sup\limits_{n\in\mathbb{N}}f_n(x)$ (resp. $f(x)=\inf\limits_{n\in\mathbb{N}}f_n(x)$) for any $x\in X$. Then $f$ is an $F_\sigma$-lower (resp. $F_\sigma$-upper) semicontinuous function.
\end{proposition}
\begin{proof}
	Let  $f(x)=\sup\limits_{n\in\mathbb{N}}f_n(x)$ and $f_n$'s are upper semicontinuous. Consider $\gamma<+\infty$. Then $f_n^{-1}\big((\gamma+\infty]\big)$ is an $F_\sigma$-set by Proposition~\ref{prop:Lebesgue_clas_of_semicontinuous}. Consequently,
	$f^{-1}\big((\gamma,+\infty]\big)=\bigcup\limits_{n=1}^\infty f_n^{-1}\big((\gamma+\infty]\big)$ is an $F_\sigma$-set as well. Thus, $f$ is $F_\sigma$-lower semicontinuous. The proof of the second case is analogical.
\end{proof}

\begin{proposition}\label{prop:semicontinuity_of_Lipschitz}
	Let $X$ and $Y$ be a metric space, $f:X\to Y$ be a continuous function and $r>0$. Then
	\begin{itemize}
		\item[$(i)$] $\Lip_r f:X\to[0;+\infty]$ is a lower semicontinuous function;
		\item[$(ii)$] $\lip_r f:X\to[0;+\infty]$ is an upper semicontinuous function;
		\item[$(iii)$] $(\lip_r f,\Lip_r f)$ is a pair of Hahn;
		\item[$(iv)$] $\Lip f:X\to[0;+\infty]$ is an $F_\sigma$-upper semicontinuous function;
		\item[$(v)$] $\lip f:X\to[0;+\infty]$ is a $F_\sigma$-lower  semicontinuous function;
		\item[$(vi)$] $(\lip f,\Lip f)$ is an $F_\sigma$-pair of Hahn;
		\item[$(vii)$] $\LIP f:X\to[0;+\infty]$ is an upper semicontinuous function.
	\end{itemize}
\end{proposition}
\begin{proof}
	$(i)$. Fix $r>0$. Let $x_0\in X$ and $\gamma<\Lip_r f(x_0)$.  Then
	$$
	\sup\limits_{\varrho<r}\Lip^\varrho f(x_0)=\Lip_r f(x_0)>\gamma.
	$$
	So, there is $\varrho\in(0;r)$ such that $\Lip^\varrho f(x_0)>\gamma$. 
    Pick $\gamma_1$ such that ${\gamma<\gamma_1<\Lip^\varrho f(x_0)}$. 
	Therefore,
	$$
	\sup\limits_{u\in B(x_0,\varrho)}\big|f(u)-f(x_0)\big|_Y=\varrho\Lip^\varrho f(x_0)>\gamma_1 \varrho.
	$$
	Thus, there is $u\in B(x_0,\varrho)$ with
	$$
	\big|f(u)-f(x_0)\big|_Y>\gamma_1 \varrho.
	$$
    Then we choose $\varrho_1$ such that $\varrho<\varrho_1<r$ and 
    $\gamma \varrho_1<\gamma_1\varrho$.
	By the continuity of $f$ at $x_0$ there exists $\delta>0$ such 
    that $\varrho+\delta<\varrho_1$ and 
    $$\big|f(x)-f(x_0)\big|_Y<\gamma_1\varrho-\gamma \varrho_1\ \ \ \text{ for any }\ \ \ x\in U=B(x_0,\delta).$$ 
	Consider $x\in U$. Then 
	$$|u-x|_X\le |u-x_0|_X+|x_0-x|_X<\varrho+\delta<\varrho_1,$$
	and, so, $u\in B(x,\varrho_1)$. Consequently,
	$$
	\big|f(u)-f(x)\big|_Y\ge \big|f(u)-f(x_0)\big|_Y-\big|f(x)-f(x_0)\big|_Y>\gamma_1\varrho-(\gamma_1\varrho-\gamma \varrho_1)=\gamma \varrho_1.
	$$
	Hence, $\Lip^{\varrho_1}f(x)>\gamma$. But $0<\varrho_1<r$. Therefore, $\Lip_r f(x)>\gamma$ for any $x\in U$. Thus, $\Lip_r f$ is lower semicontinuous at $x_0$.

	$(ii)$. Fix $r>0$. Let $x_0\in X$ and $\gamma>\lip_r f(x_0)$.  Then
	$$
	\inf_{\varrho<r}\Lip^\varrho f(x_0)=\lip_r f(x_0)<\gamma.
	$$
	So, there is $\varrho<r$ such that $\Lip^\varrho f(x_0)<\gamma$. Pick  $\gamma_1$ such that $\Lip^\varrho f(x_0)<\gamma_1<\gamma$. Then we chose $\varrho_1$ such that $0<\varrho_1<\varrho$ and $\gamma \varrho_1>\gamma_1\varrho$.
	Therefore,
	$$
	\sup\limits_{u\in B(x_0,\varrho)}\big|f(u)-f(x_0)\big|_Y=\varrho\Lip^\varrho f(x_0)<\gamma_1 \varrho.
	$$
	Then 
	$$
	\big|f(u)-f(x_0)\big|_Y<\gamma_1 \varrho\ \ \ \text{ for any }\ \ \  u\in  B(x_0,\varrho).
	$$
	By the continuity of $f$ at $x_0$ there exists $\delta>0$ such that $\varrho_1+\delta<\varrho$ and $$\big|f(x)-f(x_0)\big|_X<\gamma \varrho_1-\gamma_1 \varrho\ \ \ \text{ for any }\ \ \ x\in U=B(x_0,\delta).$$
	Consider $x\in U$ and $u\in B(x,\varrho_1)$. 
	Then
	$$|u-x_0|_X\le |u-x|_X+|x-x_0|_X<\varrho_1+\delta<\varrho,$$
	and so, $u\in B(x_0,\varrho)$. Therefore,
	$$
	\big|f(u)-f(x)\big|_Y\le \big|f(u)-f(x_0)\big|_Y+\big|f(x_0)-f(x)\big|_Y<\gamma_1\varrho+(\gamma \varrho_1-\gamma_1 \varrho)=\gamma \varrho_1.
	$$
	Thus, $\frac1{\varrho_1}\big|f(u)-f(x)\big|_Y\le \gamma$ for any $u\in B(x,\varrho_1)$.
	Hence, $\Lip^{\varrho_1}f(x)\le\gamma$. But $0<\varrho_1<r$. Therefore, $\lip_r f(x)\le\gamma$ for any $x\in U$. Thus, $\lip_r f$ is upper semicontinuous at $x_0$.
	
	$(iii)$.  It is implied from $(i)$ and $(ii)$.
	
	$(iv)$, $(v)$, $(vi)$. By (\ref{eqn:Lip_monotone}), (\ref{eqn:Lip_as_limsup}) and (\ref{eqn:lip_as_sup_lip_r})  we conclude that $\Lip_{\frac1n}f(x)\downarrow\Lip f(x)$ and $\lip_{\frac1n}f(x)\uparrow\lip f(x)$ for any $x\in X$. Thus, the needed assertions is implied from $(i)$, $(ii)$ and Proposition~\ref{prop:inf_and_sup_of_semicontinuous}.
	
	$(vii)$. Fix $x_0\in X$ and $\gamma>\LIP f(x_0)$. 
    Since $\LIP f(x_0)=\inf\limits_{r>0}\LIP^rf(x_0)$, 
    there exists $r>0$ such that $\LIP^r f(x_0)<\gamma$. 
    Set $\varrho=\frac{r}{2}$ and consider $x\in B(x_0,\varrho)$. 
    Then $B(x,\varrho)\subseteq B(x_{0},r)$. Consequently, 
	\begin{align*}
		\LIP f(x)&\le\LIP^\varrho f(x)=\sup\limits_{u\ne v\in B(x,\varrho)}\tfrac{1}{|u-v|_X}\big|f(u)-f(v)\big|_Y\\
		&\le
		\sup\limits_{u\ne v\in B(x_0,r)}\tfrac{1}{|u-v|_X}\big|f(u)-f(v)\big|_Y
		=\LIP^r f(x_0)<\gamma
	\end{align*}
	and, hence, $\LIP f$ is upper semicontinuous.
\end{proof}

Proposition~\ref{prop:Lebesgue_clas_of_semicontinuous} and \ref{prop:semicontinuity_of_Lipschitz} yield the following assertions.

\begin{proposition}\label{prp:L(f)isFsl(f)idFsd}
	Let $X$ and $Y$ be a metric space, $f:X\to Y$ be a continuous function, $x\in X$, $\varepsilon>0$ and $\gamma\ge 0$. Then
	\begin{itemize}
		\item[$(i)$] $L(f)$ is an $F_\sigma$-set;
		\item[$(ii)$] $L^\infty(f)$ is a $G_\delta$-set;
		\item[$(iii)$] $\ell(f)$ is a $G_{\delta\sigma}$-set;
		\item[$(iv)$] $\ell^\infty(f)$ is an $F_{\sigma\delta}$-set;
		\item[$(v)$] $\mathbb{L}(f)$ is an open set;
		\item[$(vi)$] $\mathbb{L}^\infty(f)$ is a closed set;
	\end{itemize}
\end{proposition}
\begin{remark} Observe that Proposition~\ref{prop:semicontinuity_of_Lipschitz}$(vii)$ and Proposition~\ref{prp:L(f)isFsl(f)idFsd}$(v)$ and $(vi)$ hold for non-continuous functions as well.
\end{remark}

\section{Takagi-van der Waerden  functions for general metric spaces}

\begin{definition}
	Let $X$ be a metric space and $\varepsilon>0$. A set $S\subseteq X$ is called 
	\begin{itemize}
		\item \textit{$\varepsilon$-separated in $X$} if for any distinct points $x,y\in S$ we have $|x-y|_X\ge\varepsilon$;
		\item \textit{maximal $\varepsilon$-separated} if it is $\varepsilon$-separated and for any $\varepsilon$-separated set $T$ in $X$ such that $S\subseteq T$ we have $S=T$;
		\item \textit{$\varepsilon$-dense} if for any $x\in X$ there is $s\in S$ with $|x-s|_X<\varepsilon$.
	\end{itemize}
\end{definition}
We always denote $d(x,A)=\inf\limits_{y\in A}|x-y|_X$ for any metric space $X$, $x\in X$ and $A\subseteq X$.
\begin{proposition}\label{prop:sep_and_dense_set}
	Let $X$ be a metric space and $\varepsilon>0$. Then
	\begin{itemize}
		\item[$(i)$] for any $\varepsilon$-separated set $S_0$ there is a maximal $\varepsilon$-separated set $S$ with $S_0\subseteq S$;
		\item[$(ii)$] for any $a>1$ there is an increasing sequence $(S_n)_{n=0}^\infty$ of maximal \mbox{$\frac{1}{a^n}$-separate}d sets $S_n$ in $X$. 
		\item[$(iii)$] an $\varepsilon$-separated set $S$ is maximal $\varepsilon$-separated if and only if $S$ is $\varepsilon$-dense;
		\item[$(iv)$] for any $\varepsilon$-dense set $S$ (in particular, for any maximal $\varepsilon$-separated set) we have that $d(x,S)<\varepsilon$ for any $x\in X$ and the function $d(\cdot,S)$ is 1-Lipschitz.
	\end{itemize}
\end{proposition}
\begin{proof} Part $(i)$ can be deduced  immediately from  the Teichm\"uller-Tukey lemma \cite[p.~9]{En}. And then we obtain $(ii)$ by induction. Part $(iii)$ can be obtained by standard considerations of maximality. Part $(iv)$ implies from the definitions and $(ii)$. 
\end{proof}

\begin{definition}\label{def:TWfunction} 
	Let $X$ be a metric space and $a>b>0$. A function $f:X\to\mathbb{R}$ is called \textit{a Takagi-van der Waerden  function of type $(a,b)$} (or \textit{TW-function of type $(a,b)$} in short) if there exist a sequence of maximal $\frac1{a^n}$-separated sets $S_n$ such that
	$$
	f(x)=\sum\limits_{n=0}^\infty b^n\varphi_n(x), \mbox{ where }\varphi_n(x)=d(x, S_n),\ \ x\in X.
	$$
	Moreover, if ${S_{n}\subseteq S_{n+1}}$ for any $n$, then $f$ is called \textit{TW-function of monotonic type $(a,b)$}.
	The \textit{$n$-th partial sum and the $n$-th remainder of $f$}  is defined as 
	$$
	s_n(x)=\sum\limits_{k=0}^{n-1} b^k\varphi_k(x),\ \ 
	r_n(x)=\sum\limits_{k=n}^\infty b^k\varphi_k(x),\ \ x\in X.
	$$
\end{definition}

In the case $S_n=\frac1{a^n}\mathbb{Z}$ and $X=\mathbb{R}$ we obtain \textit{the standard Takagi-van der Waerden  function of type $(a,b)$} $f_{a,b}:\mathbb{R}\to\mathbb{R}$,
$$
f_{a,b}(x)=\sum\limits_{n=1}^\infty b^nd\big(x, \tfrac1{a^n}\mathbb{Z}\big)
=\sum\limits_{n=1}^\infty \big(\tfrac ba\big)^n d(a^nx,\mathbb{Z}),\ \ x\in \mathbb{R}.
$$
In the case $a\in\mathbb{N}$ this function is TW-function of monotonic type $(a,b)$.
In particular, if $a=r>1$ and $b=1$ we obtain  \textit{the standard Takagi-van der Waerden  function of type $r$} $f_r:\mathbb{R}\to\mathbb{R}$,
$$
f_r(x)=f_{r,1}(x)=
\sum\limits_{n=1}^\infty \tfrac 1{r^n} d(r^nx,\mathbb{Z}),\ \ x\in \mathbb{R},
$$
which was considered, for example, in \cite{Allaart}. If we put $r=2$ or $r=10$ then we obtain the classical nowhere differentiable  \textit{Takagi function} $f_2$ or \textit{van der Waerden function} $f_{10}$ (see Introduction in \cite{Allaart} ).

\begin{proposition}\label{prop:TWfunction_is_Lipschitz}
	Let $X$ be a metric space and $a>b>0$ and $f$ is TW-function of type $(a,b)$, $n\in\mathbb{N}$, $s_n$ and $r_n$ be the $n$-th partial sum and the $n$-th reminder of $f$. Then the following conditions hold:
	\begin{itemize}
		\item[$(i)$] $f:X\to\mathbb{R}$ is a continuous function such that $0\le f(x)\le\frac a{a-b}$, $x\in X$;
		\item[$(ii)$] $r_n:X\to\mathbb{R}$ is a continuous function such that $0\le r_n(x)\le\frac {b^{n}}{(a-b)a^{n-1}}$, $x\in X$;	
		\item[$(iii)$] if $b>1$ then $s_n:X\to\mathbb{R}$ is a Lipschitz function with the constant $\frac{b^n}{b-1}$.
	\end{itemize} 
\end{proposition}
\begin{proof} 
	$(ii)$. By Proposition~\ref{prop:sep_and_dense_set}$(iv)$ we conclude that  $ b^kd(x, S_k)\le\big(\frac{b}{a}\big)^k$. Therefore, the series from  Definition~\ref{def:TWfunction} is uniformly convergent.  Thus, $r_n$ is a continuous function and
	$$0\le r_n(x)\le\sum\limits_{k=n}^\infty\big(\tfrac{b}{a}\big)^k
	=\frac{\big(\frac ba\big)^n}{1-\frac ba}=\frac {b^{n}}{(a-b)a^{n-1}}
	$$
	for any $x\in X$.

	$(i)$. Since $f=r_0$, we conclude that $(i)\Rightarrow(ii)$.
	
	$(iii)$. By Proposition~\ref{prop:sep_and_dense_set}$(iv)$ we obtain that 
	$$
	|s_n(x)-s_n(y)|\le \sum\limits_{k=0}^{n-1}b^k\big|d(x,S_k)-d(y,S_k)\big|
	\le
	$$
	$$
	\le\sum\limits_{k=0}^{n-1}b^k|x-y|_X=\frac{b^n-1}{b-1}|x-y|_X\le \frac{b^n}{b-1}|x-y|_X
	$$
	for any $x,y\in X$.
\end{proof}

\section{Big Lipschitz derivative of Takagi - van der Waerden functions}

\begin{theorem}\label{thm:BigLipTW}
	Let $X$ be a metric space, $a>b>2$ and $f$ be a TW-function of monotonic type $(a,b)$. Then $\Lip f(x)=\infty$ for any $x\in X^d$. 
\end{theorem}

\begin{proof}
	Let $S_n$, $\varphi_n$, $s_n$, $r_n$ be such as in Definition~\ref{def:TWfunction}. 
	Fix $x\in X^d$. Let us prove that $\Lip f(x)=\infty$. It is enough to prove that there exists a sequence of points $u_n\to x$ such that $u_n\ne x$ and $\frac{1}{\varrho_n}|f(u_n)-f(x)|\to\infty$ as $n\to\infty$, where  $\varrho_n=|u_n-x|_X$.

	Firstly, we consider the case  $x\in \bigcup\limits_{n=1}^\infty S_n$. Then there is $n_0\in\mathbb{N}$ such that $x\in S_{n_0}$. Since $x\in X^d$, for any $n\in\mathbb N$ there is $u_n\in B\big(x,\tfrac1{2a^n}\big)\setminus\{x\}$. Then $0<\varrho_n=|u_n-x|_X<\tfrac1{2a^n}$. In particular, $\varrho_n\to 0$. Fix $n\ge n_0$. Therefore, $x\in S_{n_0}\subseteq S_n\subseteq S_m$ and so $\varphi_m(x)=0$ for any $m\ge n$. Thus, $r_n(x)=0$. Since $S_n$ is $\frac{1}{a^n}$-separated, it is easy to see that $\varphi_n(u_n)=d(u_n,S_n)=\varrho_n$, and then $r_n(u_n)\ge b^n\varphi_n(u_n)= b^n\varrho_n$. Therefore, $|r_n(u_n)-r_n(x)|=r_n(u_n)\ge b^n\varrho_n$. On the other hand,
    Proposition~\ref{prop:TWfunction_is_Lipschitz}$(iii)$
    implies that $$|s_n(u_n)-s_n(x)|<\tfrac{b^n}{b-1}|u_n-x|_X=\tfrac{b^n\varrho_n}{b-1}.$$ Thus, 
	\begin{align*}
		|f(u_n)-f(x)|&=\Big|\big(r_n(u_n)-r_n(x)\big)+\big(s_n(u_n)-s_n(x)\big)\Big|\\
		&\ge\big|r_n(u_n)-r_n(x)\big|-\big|s_n(u_n)-s_n(x)\big|\\
		&\ge b^n\varrho_n-\tfrac{b^n\varrho_n}{b-1}=\alpha b^n \varrho_n,
	\end{align*}
	where  $\alpha=\tfrac{b-2}{b-1}>0$. Consequently, $\frac{1}{\varrho_n}|f(u_n)-f(x)|\ge \alpha b^n\to\infty$.
	
	Finally, consider the case where $x\notin\bigcup\limits_{n=1}^\infty S_n$. Fix $n\in\mathbb{N}$.  Therefore, $x\notin S_n$, and so $\varphi_n(x)=d(x,S_n)>0$. Since $\frac b2>1$,  $d(x, S_n)=\varphi_n(x)<\frac b2\varphi_n(x)$. Consequently, there is $u_n\in S_n$ with $\varrho_n=|u_n-x|_X<\frac b2\varphi_n(x)$.  Hence, $\varphi_n(x)>\frac{2\varrho_n}{b}$. 
    Proposition~{\ref{prop:sep_and_dense_set}$(iv)$} implies $\varphi_n(x)<\frac{1}{a^n}$, and so $\varrho_n\to 0$. Since $u_n\in S_n\subseteq  S_m$ for any $m\ge n$, we have that $r_n(u_n)=0$. Then 
    \begin{align*}
		\big|f(x)-f(u_n)\big|&=\big|r_n(x)+s_n(x)-s_n(u_n)\big|\\
		&\ge r_n(x)-\big|s_n(x)-s_n(u_n)\big|\\
		&\ge b^n\varphi_n(x)-\tfrac{b^n}{b-1}|u_n-x|_X\\
		&\ge b^n\tfrac{2\varrho_n}{b}-\tfrac{b^n\varrho_n}{b-1}
		=\beta b^n\varrho_n,
	\end{align*}
	where $\beta=\tfrac{b-2}{b(b-1)}$. Thus, $\frac{1}{\varrho_n}\big|f(u_n)-f(x)\big|\ge \beta b^n\to\infty$.
\end{proof}

\section{Little Lipschitz derivative of a Takagi -- van der Waerden function in a hermetic space}

In this section our goal is to obtain an analogue of Theorem~\ref{thm:BigLipTW} for the little Lipschitz derivative.
To achieve this purpose we need some special properties of a metric space. We start with some known notion of shell porosity
which was introduced in \cite{Va}.

\begin{definition}
	Let $X$ be a metric space and $x\in X$. The \textit{open shell} about $x$ of radii $r$ and $R$ is given by  $$S(x,r,R)=B(x,R)\setminus\overline{B(x,r)}.$$ 
	The \textit{shell porosity} of $X$ at $x$ is, by definition, the number
	$$
	p^s(X,x)=\limsup\limits_{r\to0^+}\tfrac1r\sup\Big\{h\in[0;r]:\exists t\in[0;r-h]\  \big|\ S(x,t,t+h)=
	\varnothing\Big\}
	$$
\end{definition}

Now we will introduce another notion, that is more suitable for our needs, but closely related to the shell porosity.
\begin{definition}
	Let $X$ be a metric space and $x\in X$. 	
        Denote
        $$
	H(X,x)=\liminf_{r\to 0^+}\tfrac1r\sup\limits_{u\in B(x,r)}|u-x|_X
	\ \text{ and }\ H(X)=\inf_{x\in X^d}H(X,x).
	$$
	The number $H(X)$ is called the \textit{hermeticity} of $X$.
        A metric space $X$ is called \textit{hermetic} if $H(X)>0$.
\end{definition}
Observe that $H(X,x)=0$ for any $x\in X\setminus X^d$. 
Obviously, we have ${0\le H(X,x)\le 1}$ for any $x\in X$. 

\begin{proposition}
	Let $X$ be a metric space and $x\in X$. Then the following equality holds $p^s(X,x)=1-H(X,x)$. In particular, $X$ is hermetic if and only if
        there is $\lambda <1$ such the ${p^s(X,x)\le\lambda}$ for any non-isolated point $x\in X$.
\end{proposition}
\begin{proof} Denote $q=1-H(X,x)$ and $p=p^s(X,x)$. Observe that
	\begin{align}
		q&=1-\liminf\limits_{r\to 0^+}\tfrac1r\sup\limits_{u\in B(x,r)}|u-x|_X \nonumber \\
		&=\limsup\limits_{r\to0^+}\tfrac1r\Big(r-\sup\limits_{u\in B(x,r)}|u-x|_X\Big) \nonumber \\
		& \label{eq:qlimsup}  =\limsup\limits_{r\to0^+}\tfrac1r\sup\big\{h\in[0;r]:S(x,r-h,r)=
		\varnothing\big\}. 
	\end{align}
Consequently, $q\le p$. To  prove the inverse inequality consider $\gamma<p$ and $\varepsilon>0$. Thus,
$$
p=\limsup\limits_{r\to0^+}\tfrac1r\sup\Big\{h\in[0;r]:\exists t\in[0;r-h]\  \big|\ S(x,t,t+h)=
	\varnothing\Big\}>\gamma,
$$
so, there exists $r<\varepsilon$ such that 
$$
\sup\big\{h\in[0;r]:\exists t\in[0;r-h]\  \big|\ S(x,t,t+h)=
\varnothing\big\}>\gamma r.
$$
So, there are $h\le r$ and $t\le r-h$ such that $S(x,t,t+h)=\varnothing$ and $h>\gamma r$. 
Put $r_1=t+h$. Then $r_1\le r<\varepsilon$ and $S(x,r_1-h,r_1)=S(x,t,t+h)=\varnothing$. So, $h>\gamma r\ge \gamma r_1$. 
Thus, by \eqref{eq:qlimsup} $q\ge \gamma$. Therefore, $q=p$. 
\end{proof}
If we consider the function $d_x(u)=|u-x|_X$, $u\in X$, then 
$$
H(X,x)=\lip d_x(x)=\lim_{r\to0^+}\lip_r d_x(x).
$$
For any $\lambda<H(X,x)$ we  define the \emph{radius of hermeticity} as the number
$$
RH_\lambda(X,x)=\sup\limits\big\{r>0:\lip_r d_x(x)>\lambda\big\}.
$$
Clearly, $RH_\lambda(X,x)>0$.

The proof of the following propositions are straightforward.
\begin{proposition}\label{prop:SmallLipTW1}
	Let $X$ be a metric space, $x\in X$ and $\lambda\ge 0$. If $\lambda<H(X,x)$ then for any $0\le r<RH_\lambda(X,x)$ there exists $u\in X$ such that 
	\begin{equation}\label{equ:SmallLipTW0}
		\lambda r\le |x-u|_X\le r
	\end{equation}
	Conversely, if there is $\varepsilon>0$ such that for any $0\le r<\varepsilon$ there is $u\in X$ with (\ref{equ:SmallLipTW0}) then $H(X,x)\ge\lambda$ and $RH_\lambda(X,x)\ge\varepsilon$.
\end{proposition}

\begin{proposition}\label{prop:SmallLipTW2}
	Let $M$ be a metric subspace of some normed space $X$ such that $M\subseteq \Int\overline M\ne\varnothing$. Then, $H(M,x)=1$ for any $x\in M$. In particular, $M$ is hermetic.
\end{proposition}

\begin{lemma}\label{lem:SmallLipTW}
	Let $X$ be a metric space, $x\in X^d$, $0<\lambda<H(X,x)$, 
    $0<\varepsilon<RH_\lambda(X,x)$, $S$ be a maximal $\varepsilon$-separated 
    set in $X$ and $\varphi(\,\cdot\,)=d(\,\cdot\, ,S)$. 
    Then there exists $u\in B[x,\varepsilon]$ such that
	\[
    	|\varphi(u)-\varphi(x)|\ge \tfrac{\lambda\varepsilon}{8}.
    \]
\end{lemma}
\begin{proof}
	Set $\alpha=\frac{\lambda}{8}$ and prove that $|\varphi(u)-\varphi(x)|\ge\alpha\varepsilon$ for some $u\in B[x,\varepsilon]$.
	
	Firstly, we consider the case $\varphi(x)\ge\alpha\varepsilon$. Since $\varphi(x)=d(x,S)<\varepsilon$ by
    {Proposition~\ref{prop:sep_and_dense_set}$(iv)$}, we conclude
    that there is a point $u\in S$ such that $|u-x|_X<\varepsilon$. Therefore, $\varphi(u)=0$ and then
	$$
	\big|\varphi(u)-\varphi(x)\big|=\varphi(x)\ge\alpha\varepsilon.
	$$
	
	Secondly, we consider the case  $\varphi(x)<\alpha\varepsilon$. Then there exists $s\in S$ such that $|s-x|_X<\alpha\varepsilon$. Set $r=\frac{3\varepsilon}{8}$. Since $0<r<\varepsilon<RH_\lambda(X,x)$, there is $u\in X$ with 
	$$
	\lambda r\le |u-x|_X\le r 
	$$
	by Proposition~\ref{prop:SmallLipTW1}. Since $\lambda<H(X,x)\le 1$, we have $\alpha=\frac{\lambda}{8}<\frac 18$.
	Therefore,
	$$
	|u-s|_X\le |u-x|_X+|x-s|_X<r+\alpha\varepsilon<r+\tfrac{\varepsilon}{8}=\tfrac{3\varepsilon}{8}+\tfrac{\varepsilon}{8}=\tfrac{\varepsilon}{2}.
	$$
	Let us consider some $t\in S\setminus\{s\}$. Since $S$ is $\varepsilon$-separated, $|s-t|_X\ge\varepsilon$. Hence,
	$$
	|u-t|_X\ge|t-s|_X-|u-s|_X\ge\varepsilon-\tfrac{\varepsilon}{2}=\tfrac{\varepsilon}{2}\ge |u-s|_X.
	$$
	Thus, $\varphi(u)=d(u,S)=|u-s|_X$. Therefore,
	$$
	\varphi(u)=|u-s|_X\ge|u-x|_X-|s-x|_X\ge\lambda r-\alpha\varepsilon=\tfrac{3\lambda\varepsilon}{8}-\alpha\varepsilon=3\alpha\varepsilon-\alpha\varepsilon=2\alpha\varepsilon,
	$$
	and then
	$$
	\big|\varphi(u)-\varphi(x)\big|\ge \varphi(u)-\varphi(x)\ge 2\alpha\varepsilon-\alpha\varepsilon=\alpha\varepsilon.
	$$
	Thus, in the both cases we have that  $|\varphi(u)-\varphi(x)|\ge\alpha\varepsilon$ for some point ${u\in B[x,\varepsilon]}$.
\end{proof}

\begin{theorem}\label{thm:SmallLipTW}
	Let $X$ be a hermetic metric space, $a>b>1$ such that 
	\begin{equation}\label{equ:SmallLipTW}
		\frac{2b}{a-b}+\frac{1}{b-1}<\frac{H(X)}{8},
	\end{equation}
	and $f$ be TW-function of type $(a,b)$. Then $\lip f(x)=\infty$ for any $x\in X^d$.
\end{theorem}

\begin{proof}
	Let $S_n$, $\varphi_n$, $s_n$, $r_n$ be such as in Definition~\ref{def:TWfunction}. 
	Fix $x\in X^d$. Let us prove that $\lip f(x)=\infty$.
	Since (\ref{equ:SmallLipTW}), we can pick  $\lambda>0$ such that $\lambda<H(X)$ and 
	\begin{equation}\label{equ:SmallLipTW1}
		\gamma:=\frac{\lambda}{8}-\frac{2b}{a-b}-\frac{1}{b-1}>0.
	\end{equation}
	Therefore, $RH_\lambda(X,x)>0$.	Let $\delta=\min\big\{1,RH_\lambda(X,x)\big\}$ and fix $0<r<\delta$. Set $\varepsilon_n=\frac{1}{a^n}$. Since $\varepsilon_0=1\ge\delta$ and $\varepsilon_n\downarrow 0$, there exists $n=n(r)\in\mathbb{N}$ such that $\varepsilon_n\le r<\varepsilon_{n-1}$. Put $\varepsilon=\varepsilon(r)=\varepsilon_n$. Observe that  $\varepsilon(r)\to 0$ and then $n(r)\to\infty$  as $r\to 0^+$. Since $\varepsilon=\varepsilon_n\le r<\delta\le RH_\lambda(X,x)$ we can us Lemma~\ref{lem:SmallLipTW} for the set $S=S_n$ and the function $\varphi=\varphi_n$. So, there is $u\in B[x,\varepsilon]$ with 
	$$
	\big|\varphi(u)-\varphi(x)\big|\ge\tfrac{\lambda\varepsilon}{8}.
	$$
	Denote $g=s_n$ and $h=r_{n+1}$. Then
	$$
	f=g+b^n\varphi+h
	$$
	By Proposition~\ref{prop:TWfunction_is_Lipschitz} the function $g$ is Lipschitz with the constant $\frac{b^n}{b-1}$ and $0\le h\le\frac {b^{n+1}}{(a-b)a^{n}}$. Therefore,
	$$
	\big|g(u)-g(x)\big|\le \tfrac{b^n}{b-1}|u-x|_X\le\tfrac{1}{b-1}b^n\varepsilon,
	$$
	$$
	\big|h(u)-h(x)\big|\le\big|h(u)\big|+\big|h(x)\big|\le \tfrac {2b^{n+1}}{(a-b)a^{n}}
	=\tfrac{2b}{a-b}b^n\varepsilon,
	$$
	$$
	\big|b^n\varphi(u)-b^n\varphi(x)\big| \geq  b^n\tfrac{\lambda \varepsilon}8=\tfrac{\lambda}{8}b^n\varepsilon.
	$$
	Thus,
	$$
	\big|f(u)-f(x)\big|\ge \big|b^n\varphi(u)-b^n\varphi(x)\big| - \big|g(u)-g(x)\big| -\big|h(u)-h(x)\big|
	\ge
	$$
	$$
	\ge  \tfrac{\lambda}{8}b^n\varepsilon-\tfrac{1}{b-1}b^n\varepsilon-\tfrac{2b}{a-b}b^n\varepsilon
	=\Big(\tfrac{\lambda}{8}-\tfrac{1}{b-1}-\tfrac{2b}{a-b}\Big)b^n\varepsilon=\gamma b^n\varepsilon=
	$$
	$$
	=\gamma b^n\varepsilon_n=\tfrac{\gamma b^n}{a^n}=\tfrac{\gamma }{a}b^n\varepsilon_{n-1}\ge\tfrac{\gamma}{a} b^n r.
	$$
	Since, $|u-x|_{X}\le \varepsilon < r$ and $\frac{1}{r}\vert f(u)-f(x)\vert \geq \tfrac{\gamma}{a} b^n$, we conclude that
	$$
	\Lip^r f(x)\ge\tfrac{\gamma}{a}b^n, \mbox{ for any } r<\delta.
	$$
	But $n=n(r)\to\infty$ as $r\to0^+$ and $b>1$. Thus,
	$$
	\lip f(x)=\liminf_{r\to 0^+}\Lip^r f(x)\ge\lim_{r\to0^+}\tfrac{\gamma}{a}b^{n(r)}=\infty.
	$$
	So, $\lip f(x)=\infty$ for any $x\in X^d$.
\end{proof}

\section{$\ell$-$L$-$\mathbb{L}$-problem for an open set}

In this section we give a partial solution of
Problem \ref{prob:ellLbL1} for the case
where $A=B$ is an open set and $C=\overline{A}$.

\begin{lemma}\label{lem:construction_of_g}
	Let $X$ be a metric space, $G$ be an open subset of $X$, $F=X\setminus G$  and $\alpha>0$. Then there exist $\alpha$-Lipschitz functions $g:X\to[0;\alpha]$ such that $F=g^{-1}(0)$.
\end{lemma}
\begin{proof} 
	If $F=\varnothing$ then we put $g(x)=\alpha$ for any $x\in X$. Suppose that $F\ne\varnothing$. 
	Let us define a function $g$ as
	$$
	g(x)=\alpha\min\big\{1,d(x,F)\big\},\ \ \ x\in X.
	$$
        Consider $x,y\in X$. Since 
	$d(x,F)\le |x-y|_X+d(y,F)$,
        we conclude
	\begin{align*}
		g(x)&=\alpha\min\big\{1,d(x,F)\big\}
		\le
		\alpha\min\big\{1,|x-y|_X+d(y,F)\big\}\\
		&\le
		\alpha|x-y|_X+\alpha\min\big\{1,d(y,F)\big\}
		=\alpha|x-y|_X+g(y).
	\end{align*}
	Replacing $x$ with $y$ we obtain that $g(y)\le\alpha|x-y|_X+g(x)$. Thus,
	$$
	\big|g(x)-g(y)\big|\le\alpha|x-y|_X,\ \ \ \text{for any }x,y\in X,
	$$
	and then $g$ is $\alpha$-Lipschitz. On the other hand, since $F$ is closed,
	$$
	g(x)=0\Leftrightarrow d(x,F)=0\Leftrightarrow x\in F, \ \ \ x\in X.
	$$
	Hence, $F=g^{-1}(0)$, and then $g$ is such as we need.
\end{proof}

\begin{theorem} 
Let $X$ be a hermetic metric space and $G\subseteq X^d$ be an open set in $X$. 
Then there exists a continuous function $f:X\to \mathbb{R}$ such that $L^\infty(f)=\ell^\infty(f)=G$ and $\mathbb{L}^\infty(f)=\overline{G}$.
\end{theorem}
\begin{proof} Since
	$$
	\lim_{b\to\infty}\lim_{a\to\infty}\Big(\frac{2b}{a-b}+\frac{1}{b-1}\Big)=0<\frac{H(X)}{8},
	$$
	we can chose $b>1$ and then $a>b$ such that inequality	(\ref{equ:SmallLipTW}) holds. Let $h=f_{a,b}$ be a TW-function of type $(a,b)$. By Theorem~\ref{thm:SmallLipTW} and  Proposition~\ref{prop:TWfunction_is_Lipschitz} we have that $\lip h(x)=\infty$ and $0\le h(x)\le\frac{1}{\alpha}$ 
    for any $x\in X$, where $\alpha=\frac{a-b}{a}$. Let $F=X\setminus G$ and $U=\Int F=X\setminus \overline{G}$. 
    Let $g$ be given by Lemma~\ref{lem:construction_of_g} and define the function $f:X\to[0;1]$ by $f(x)=g(x)h(x)$ for any $x\in X$.
	
    Consider $x\in F$ and $u\in X$. Therefore,
    \begin{align*}
    	\big|f(u)-f(x)\big|&=\big|g(u)h(u)-g(x)h(x)\big|=g(u)h(u)\le\tfrac1\alpha g(u)\\
    	&=\tfrac1\alpha\big(g(u)-g(x)\big)\le \tfrac1\alpha\cdot\alpha\big|u-x\big|_X=\big|u-x\big|_X,
    \end{align*}
    	and so, $\Lip f(x)\le 1$. Hence, $\ell^{\infty}(f)\subseteq L^{\infty}(f)\subseteq X\setminus F=G$.
        On the other hand, if $x\in U$, then the above inequality yields $\LIP f(x)\le 1$. 
        Therefore, $\mathbb L^\infty(f)\subseteq X\setminus U=\overline{G}$.
    	
    Now we consider $x\in G$ and we want to prove $\lip f(x)=+\infty$. 
    For any $r>0$ and $u\in B(x,r)$ we have that
    \begin{align*}
    	\big|f(u)-f(x)\big|&=\big|g(u)h(u)-g(x)h(x)\big|\\
    	&=\Big| g(x)\big(h(x)-h(u)\big)-h(u)\big(g(u)-g(x)\big)\Big|\\
    	&\ge g(x)\big|h(x)-h(u)\big|-h(u)\big|g(u)-g(x)\big|\\
    	&\ge g(x)\big|h(x)-h(u)\big|-\tfrac{1}{\alpha}\cdot \alpha|x-u|\\
    	&\ge g(x)\big|h(x)-h(u)\big|-r.
    \end{align*}
    Hence,
    \begin{align*}
    	\Lip^r f(x)&=\sup\limits_{u\in B(x,r)}\tfrac1r\big|f(u)-f(x)\big|\\
    	&\ge g(x)\sup\limits_{u\in B(x,r)}\tfrac1r\big|h(u)-h(x)\big|-1\\
    	&=g(x)\Lip^rh(x)-1.
    \end{align*}
    Consequently, since $g(x)>0$, we conclude that
    \begin{align*}
    	\lip f(x)&=\liminf_{r\to 0^+}\Lip^r f(x)\\
    	&\ge g(x)\liminf_{r\to 0^+}\Lip^r h(x)-1\\
    	&=g(x)\lip h(x)-1=+\infty. 
    \end{align*}
    Hence, $\lip f(x)=\infty$ on $G$, that is 
    $G\subseteq\ell^\infty(f)$. 
    We have already proven that $\ell^{\infty}(f)\subseteq L^{\infty}(f)\subseteq G$ and
    $\mathbb L^\infty(f)\subseteq\overline{G}$. So, $G=\ell^{\infty}(f)=L^{\infty}(f)\subseteq\mathbb{L}^{\infty}(f)$.
    Since $\mathbb{L}^{\infty}(f)$ is closed, we also have $\mathbb{L}^{\infty}(f)=\overline{G}$.
\end{proof}

Using Theorem~\ref{thm:BigLipTW} instead of Theorem~\ref{thm:SmallLipTW} in the previous  proof  we obtain the following result.

\begin{theorem} Let $X$ be a metric space without isolated point  and $G$ be an open subset of $X$. Then there exists a continuous function $f:X\to \mathbb{R}$ such that $L^\infty(f)=G$ and $\mathbb{L}^\infty(f)=\overline{G}$.
\end{theorem}

\section*{Acknowledgments}
The authors would like to appreciate Thomas Zürcher, whose talk on  the Seminar of Real Analysis in Katowice inspired us to engage in this topic.




\end{document}